\documentclass[12pt,reqno]{amsart}
\usepackage{amssymb}
\usepackage{mathrsfs}
\usepackage{a4}
\usepackage{amscd}
\usepackage{tikz}
\usepackage{hyperref}
\usetikzlibrary{decorations.markings}

\newcommand{\heute}{7 May 2015}

\theoremstyle{plain}
\newtheorem{theorem}{Theorem}[section]

\newtheorem{lemma}[theorem]{Lemma}
\newtheorem{corollary}[theorem]{Corollary}
\newtheorem{proposition}[theorem]{Proposition}

\theoremstyle{remark}
\newtheorem{remark}[theorem]{Remark}
\newtheorem{notation}[theorem]{Notation}

\newtheorem*{defn}{Definition}
\newtheorem*{rk}{Remark}
\newtheorem*{bsp}{Example}
\newtheorem*{notn}{Notation}

\newcommand{\bulit}{\item[$\bullet$]}

\newcommand{\enref}[1]{\textup{(\ref{enum:#1})}}
\newcommand{\dashTwo}[1]{\textup{(\ref{two}${}'$)}}

\newcommand{\ignore}[1]{}
\newcommand{\f}[1][p]{\mathbb{F}_{#1}}

\newcommand{\qq}{\mathbb{Q}}
\newcommand{\Gro}[1]{Gr\"ob\-ner}

\newcommand{\Ann}{\operatorname{Ann}}
\newcommand{\Aut}{\operatorname{Aut}}

\newcommand{\Hom}{\operatorname{Hom}}
\newcommand{\Id}{\operatorname{Id}}

\newcommand{\eps}{\varepsilon}
\DeclareMathOperator{\Bild}{Im}

\newcommand{\N}{\mathbb{N}}

\newcommand{\Z}{\mathbb{Z}}

\newcommand{\abs}[1]{\left|#1\right|}

\newcommand{\Aa}{\mathscr{A}}
\newcommand{\zz}{\mathbb{Z}}

\newcommand{\cc}{\mathbb{C}}
\newcommand{\bb}{\mathbb{B}}

\newcommand{\cbdry}{\Delta}

\newcommand{\coySeq}[1]{cocycle sequence}
\newcommand{\cohSeq}[1]{cochain sequence}
\newcommand{\CohSeq}[1]{Cochain sequence}

\DeclareMathOperator{\pro}{pro}
\DeclareMathOperator{\inc}{inc}
\DeclareMathOperator{\mul}{mul}
\DeclareMathOperator{\con}{con}

\tikzset{sArrow/.style={decoration={markings, mark=at position .5 with {\arrow{>}}}}}

\begin{document}

\title[Cochain sequences and the Quillen category]
      {Cochain sequences and the Quillen category of a coclass family}
\author[B.~Eick]{Bettina Eick}
\address{Institut Computational Mathematics \\
         TU Braunschweig \\ 38106 Braunschweig \\ Germany }
\email{beick@tu-bs.de}
\author[D.~J. Green]{David J. Green}
\address{Institute for Mathematics \\
         Friedrich-Schiller-Universit\"at Jena \\ 07737 Jena \\ Germany}
\email{david.green@uni-jena.de}
\thanks{Green received travel assistance from DFG grant GR 1585/6-1\@.}
\date{\heute}

\begin{abstract}
We introduce the concept of an infinite cochain sequence and initiate
a theory of homological algebra for them. We show how these
sequences simplify and improve the construction of infinite coclass 
families (as introduced by Eick and Leedham-Green) and how they apply in
proving that almost all groups in such a family have equivalent Quillen 
categories. We also include some examples of infinite families of $p$-groups
from different coclass families that have equivalent Quillen categories.
\end{abstract}

\maketitle

\section{Introduction}

\noindent
Coclass theory was initiated by Leedham-Green and Newman
\cite{LeedhamGreenNewman}. The fundamental aim of this theory is to classify
and investigate finite $p$-groups using the coclass as primary invariant.
The infinite coclass families of finite $p$-groups of fixed coclass as 
defined by Eick and Leedham-Green \cite{ELG08} are considered a step towards
these aims. Their definition is based on a splitting theorem for a certain 
type of second cohomology group.

Various interesting properties of the infinite coclass families have been 
determined. For example, the automorphism groups and the Schur multiplicators
of the groups in one family can be described simultaneously for all groups 
in the family, see \cite{Eic06b,Cou} and \cite{EFe10b}. It is conjectured 
that almost all groups in an infinite coclass family have isomorphic 
mod-$p$ cohomology rings. This conjecture is still open, but it is 
underlined by computational evidence obtained by Eick and 
§King~\cite{EickKing:IsoProblem} and by our earlier result \cite{quillen} 
saying that the Quillen categories of almost all groups in an infinite 
coclass family are equivalent. The proof of the latter theorem uses a 
splitting theorem for cohomology groups.
\medskip

\noindent
In this paper we derive a generalization of the splitting theorems
obtained and used in \cite{ELG08} and \cite{quillen},
and describe the splitting at the cocycle level.
Based on this, we introduce the concept of an infinite cochain
sequence and take the first steps towards the development of a theory of
homological algebra for them.

We then show that the infinite coclass families of~\cite{ELG08} can be defined
using the infinite cochain sequence. This way of defining the families is
more explicit than the definition in \cite{ELG08}, since it is based on
cocycles rather than just cohomology classes. This difference is significant; for
example, it is useful for the investigation of the Quillen categories of 
the groups in a coclass family. Further, we use the infinite cochain 
sequences to give a new, more conceptual proof of our main theorem in 
\cite{quillen} on the Quillen categories of the groups in an infinite 
coclass family. 

In final part of this paper we give some examples of groups from different 
coclass families with equivalent Quillen categories.  Let $q = p^{\ell}$ for
a prime $p$, let $\Z_p$ denote the $p$-adic integers and consider the 
irreducible action of $C_q$ on $T = \Z_p^{(p-1)p^{\ell-1}}$ (this is unique 
up to equivalence). Then $G_q = T \rtimes C_q$ is an infinite pro-$p$-group 
of coclass $\ell$. For $\ell = 1$ it is the unique infinite pro-$p$-group 
of coclass $1$ (or of maximal class) and for $\ell > 1$ it is an interesting 
example for an infinite pro-$p$-group of coclass $\ell$. 

The main line groups associated with an infinite pro-$p$-group $G$ of
coclass $r$ are the infinitely many lower central series quotients 
$G / \gamma_i(G)$ that have coclass $r$; this infinite sequence is not
necessarily a coclass family itself, but it consists of finitely many
different coclass families. The skeleton groups associated with 
an infinite pro-$p$-group $G$ of coclass $r$ form a significantly 
larger family of groups containing the main line groups and they play
an important role in coclass theory; we refer to 
\cite[Sec. 8.4]{LeedhamGreenMcKay:book}~or 
\cite[Sec. 3]{E-LG-N-OB} for details.

\begin{theorem} \ 
\begin{enumerate}
\item \label{enum:examplesThm-1}
For a arbitrary, fixed prime $p$, the Quillen categories of almost all 
main line groups associated with $G_p$ are pairwise equivalent.
\item \label{enum:examplesThm-2}
The Quillen categories of almost all skeleton groups associated with 
$G_9$ are pairwise equivalent.
\end{enumerate}
\end{theorem}

\begin{proof}
\enref{examplesThm-1}: See Section~\ref{subsection:Gp} for odd~$p$; 
for $p=2$ the main line groups are the dihedral $2$-groups, and the 
result is well known, see e.g.\@ \cite[Sec. 9]{quillen}\@.
\item \enref{examplesThm-2}: See Section~\ref{subsection:G9}\@.
\end{proof}

\begin{rk}
Theorem \ref{enum:examplesThm-1} (1) can be made more explicit: the Quillen
categories of the quotients $G_p/\gamma_i(G_p)$ are equivalent for all
$i \geq p+1$. Note that $G_3/\gamma_i(G_3)$ have isomorphic mod-$3$ 
cohomology rings for $i \in \{5,6,7\}$, but the cohomology ring for $i=4$ is different,
see~\cite{EickKing:IsoProblem}.
\end{rk}

\section{Infinite cochain sequences}
\label{prelim-cohom}

\subsection{Preliminaries}

We work throughout with the normalized standard resolution, 
see~\cite[p.~8]{Evens:book}\@. That is, a cochain $f \in C^n(G,M)$ is a 
map $f \colon G^n \rightarrow M$ with the additional property that 
$f(g_1,\ldots,g_n)$ is zero if any $g_i$ is the identity element.

We denote the coboundary operator by $\cbdry \colon C^n(G,M) \rightarrow 
C^{n+1}(G,M)$. Recall that the coboundary of a $2$-cochain is given by
\[
\cbdry f(g_1,g_2,g_3) = {}^{g_1} f(g_2,g_3) - f(g_1g_2,g_3) + f(g_1,g_2g_3) - 
f(g_1,g_2)
\]
and more generally the coboundary of an $n$-cochain is
\begin{align*}
\cbdry f(g_1,\ldots,g_{n+1}) & = {}^{g_1} f(g_2,\ldots,g_{n+1}) \\ 
& \quad {} + \sum_{i=1}^n (-1)^i f(g_1,\ldots,g_{i-1},g_i g_{i+1},
g_{i+2},\ldots,g_{n+1}) \\ & \quad {} + (-1)^{n+1} f(g_1,\ldots,g_n) \, .
\end{align*}
The $n$-cocycles are the elements of
\[
Z^n(G,M) = \ker\bigl(C^n(G,M) \xrightarrow{\cbdry} C^{n+1}(G,M)\bigr) \, ,
\]
and the $n$-coboundaries are the elements of
\[
B^n(G,M) = \Bild\bigl(C^{n-1}(G,M) \xrightarrow{\cbdry} C^n(G,M)\bigr) \, .
\]
Since $\cbdry^2 = 0$ it follows that $B^n(G,M) \subseteq 
Z^n(G,M)$, and we set $H^n(G,M) = Z^n(G,M)/B^n(G,M)$. Elements of $H^n(G,M)$ 
are called cohomology classes; if $f$ is an $n$-cocycle, then its cohomology 
class is $f + B^n(G,M) \in H^n(G,M)$.

\begin{remark}
\label{remark:transfer}
By transfer theory, $\abs{G} \cdot H^n(G,M) = 0$ for all $n \geq 1$: see e.g.\@ \cite[Proposition 3.6.17]{Benson:I}\@.
\end{remark}

\subsection{Splitting theorems}

From now on for the remainder of this section we fix the following notation.

\begin{notation}
\label{setup}
Let $G$ be a finite $p$-group with $m = \log_p(|G|)$, let $R$ be a 
commutative ring, let $M$ an $RG$-module and let $N$ be a submodule of
$M$ with $\Ann_N(p) = \{0\}$. 
\end{notation}

We need the following generalization of \cite[Theorem~7]{quillen}, which is 
itself a generalization of \cite[Theorem~18]{ELG08}. 

\begin{theorem}
\label{thm18}
We use Notation \ref{setup} and let $n \geq 1$ and $r \geq 2m$. Then there 
is a splitting
\[ H^n(G,M/p^rN) \cong H^n(G,M) \oplus H^{n+1}(G,N) \]
which is natural with respect to restriction to subgroups of~$G$.
\end{theorem}

\begin{notn}
Projection $M \twoheadrightarrow M/p^r N$ induces maps $C^n(G,M) 
\twoheadrightarrow C^n(G,M/p^r N)$ of cochain modules and $H^n(G,M) 
\rightarrow H^n(G,M/p^r N)$ of cohomology modules. We shall denote 
all three maps by $\pro_r$.
\end{notn}

\begin{proof}
Recall that $\abs{G} = p^m$ and define $i_r \colon N \rightarrow M : x 
\mapsto p^r x$. Consider the long exact sequence in group cohomology
induced by the following short exact sequence of coefficient modules
\[
\begin{CD}
0 @>>> N @>{i_r}>> M @>{\pro_r}>> M/p^r N @>>> 0 \\
\end{CD} \, .
\]
The proof of \cite[Theorem~7]{quillen} readily generalizes, showing that 
if $n \geq 0$ and $r \geq 2m$ then
\[
\begin{CD}
0 @>>> H^n(G,M) @>{\pro_r}>> H^n (G, M/p^rN) @>{\con_r}>> H^{n+1}(G,N) @>>> 0 \\
\end{CD}
\]
is a split short exact sequence, where $\con_r$ is the connecting homomorphism.
\end{proof}

\noindent
We now describe how Theorem~\ref{thm18} works at the cocycle level.

\begin{proposition}
\label{prop:splittingCocycleLevel}
We use Notation \ref{setup} and let $n \geq 1$, pick $\rho \in Z^n(G,M)$ 
and $\eta \in Z^{n+1}(G,N)$. 
\begin{enumerate}
\item
\label{enum:sCL-1}
There is a (not necessarily unique) $n$-cochain $\sigma \in C^n(G,N)$ 
such that $\cbdry(\sigma) = p^m \eta$.
\item
\label{enum:sCL-2}
For every $r \geq m$ and for every choice of $\sigma$ in~\enref{sCL-1},
the induced cochain $\pro_r(\rho + p^{r-m} \sigma)$ lies in $Z^n(G,M/p^rN)$. 
\item
\label{enum:sCL-3}
For every $r \geq 2m$ and for every choice of $\sigma$ in~\enref{sCL-1}, the cohomology class
\[
\pro_r(\rho + p^{r-m} \sigma) + B^n(G, M/p^rN) \in H^n(G, M/p^r N)
\]
is the unique class corresponding via the isomorphism
of Theorem~\ref{thm18} to 
\[
\left(\rho + B^n(G,M),\eta+B^{n+1}(G,N) \right) 
  \in H^n(G,M) \oplus H^{n+1}(G,N) \, .
\]
\end{enumerate}
\end{proposition}

\begin{proof}
\enref{sCL-1}: $p^m \eta$ is a coboundary, since $p^m H^{n+1}(G,N) = 0$ by Remark~\ref{remark:transfer}\@.
\item \enref{sCL-2}: $\pro_r$ and $\cbdry$ commute, and $\cbdry(\rho + p^{r-m} \sigma) = p^r \eta$ lies in the kernel of $\pro_r$.
\item
\enref{sCL-3}: 
The proof of \cite[Theorem~7]{quillen} says that the component maps of the isomorphism $H^n(G,M/p^r N) \rightarrow H^n(G,M) \oplus H^{n+1}(G,N)$ are the connecting homomorphism
$\con_r \colon H^n (G, M/p^rN) \rightarrow H^{n+1}(G,N)$ and the map
\[
H^n(G,M/p^r N) \xrightarrow{\pi_*} H^n(G,M/p^{r-m}N) \xrightarrow{(\pro_{r-m})^{-1}} H^n(G,M) \, ,
\]
where $\pi \colon M/p^rN \rightarrow M/p^{r-m}N$ is the projection map $x+p^r N \mapsto x + p^{r-m} N$. As
\[
\pi_* \pro_r (\rho + p^{r-m} \sigma) = \pro_{r-m}(\rho + p^{r-m} \sigma) = \pro_{r-m}(\rho) \, ,
\]
the image in $H^n(G,M)$ is $\rho + B^n(G,M)$.

Recall from e.g.~the proof of \cite[Theorem 9.1.5]{LeedhamGreenMcKay:book} that $\con_r$ is constructed as the composition
\[
Z^n(G,M/p^r N) \xrightarrow{(\pro_r)^{-1}} C^n(G,M) \xrightarrow{\Delta} Z^{n+1}(G,M) \xrightarrow{(i_r)_*^{-1}} Z^{n+1}(G,N) \, ,
\]
with $i_r$ as in the proof of Theorem~\ref{thm18}\@.
So $\pro_r(\rho + p^{r-m} \sigma) \mapsto \rho + p^{r-m} \sigma \mapsto 0 + p^r \eta = i_r(\eta) \mapsto \eta$. Uniqueness follows.
\end{proof}

\subsection{The definition of \cohSeq.s}
\label{CohSeqs-def}

\noindent
Using the ideas of Proposition~\ref{prop:splittingCocycleLevel} we now define
infinite \cohSeq.s. 

\begin{defn}
We use Notation \ref{setup} and let $n \geq 1$ and $r_0 \geq 0$. 
We call a sequence $(\alpha_r)_{r \geq r_0}$ of cochains $\alpha_r \in 
C^n(G, M/p^r N)$ a \emph{\cohSeq.} if there are cochains $\rho \in C^n(G,M)$ 
and $\sigma \in C^n(G,N)$ and an $\omega \in \{0,1,\ldots, r_0\}$ such that
\[
\alpha_r = \pro_r(\rho + p^{r-\omega} \sigma)  \in C^n(G, M/p^r N) \quad 
      \text{for all $r \geq r_0$.}
\]
\end{defn}

Note that the \cohSeq.s defined by $(\rho, \sigma; r_0,\omega)$ and 
$(\rho', \sigma'; r'_0,\omega')$ are equal if and only if $r_0 = r'_0$ and
the cochains $\pro_r(\rho + p^{r-\omega} \sigma)$ and $\pro_r(\rho' + 
p^{r-\omega'} \sigma')$ are equal as elements of $C^n(G, M/p^r N)$ for all
$r \geq r_0$.

Often $r_0$ will be clear from the context. We then also write 
$\alpha_{\bullet}$ for $( \alpha_r \mid r \geq r_0)$ and $M/p^\bullet N$ 
for $(M/p^rN \mid r \geq r_0)$. If $\alpha_\bullet$ is induced from
$(\rho, \sigma; r_0, \omega)$, then we also write 
\[ 
\alpha_{\bullet} = \pro_{\bullet}(\rho + p^{\bullet-\omega} \sigma) \] 

\begin{notn}
Often $r_0$ and $N$ will be fixed from the context. We then denote 
$M_r = M/p^rN$, we write $M_{\bullet}$ for $(M_r \mid r \geq r_0 )$, and
we denote with $\cc^n_{r_0}(G,M_{\bullet})$ the set of all \cohSeq.s which 
start at $r_0$.
\end{notn}

\subsection{Homological algebra for \cohSeq.s}
\label{CohSeqs}

\noindent
Our next aim is to develop some elementary homological algebra for 
\cohSeq.s. 

\begin{notation}
\label{setup2}
We continue to use the Notation~\ref{setup}, imposing minor 
additional restrictions. We assume from now on that $R$ is a noetherian 
integral domain and $p$ a prime number, which in $R$ is neither zero nor 
a unit. Further, $M$ is a finitely generated $RG$-module which is free as 
an $R$-module. Then $\bigcap_r p^r M = \{0\}$ by Krull's 
Theorem~\cite[10.17]{AtiMac}, and $\Ann_N(p) = \{0\}$.
\end{notation}

\begin{lemma}
\label{lemma:moduleCochainSeq}
The set $\cc^n_{r_0}(G,M_{\bullet})$ of \cohSeq.s is an $R$-module.
\end{lemma}

\begin{proof}
Let $\alpha_{\bullet}$ be defined by $(\rho,\sigma;r_0,\omega)$, and $\beta_{\bullet}$ by $(\rho',\sigma';r_0,\omega')$. Then
\[
\alpha_r + \beta_r = \pro_r\left(\rho + \rho' + p^{r-\ell}(p^{\ell-\omega} \sigma + p^{\ell-\omega'} \sigma')\right) \quad \text{for $\ell = \max(\omega,\omega')$,}
\]
and so $\alpha_{\bullet} + \beta_{\bullet}$ is the \cohSeq. defined by $(\rho+\rho',p^{\ell-\omega} \sigma + p^{\ell-\omega'} \sigma';r_0,\ell)$. And for $x \in R$, $x \alpha_{\bullet}$ is the \cohSeq. defined by $(x\rho;x\sigma;r_0,\omega)$.
\end{proof}

\begin{lemma}
\label{lemma:identityCochainSeq}
Let $\alpha_{\bullet} \in \cc^n_{r_0}(G,M_{\bullet})$ be the \cohSeq. defined by $(\rho;\sigma;r_0,\omega)$.
\begin{enumerate}
\item
\label{enum:iCS-1}
Either $\alpha_{\bullet} = 0$, that is $\alpha_r = 0$ for all $r \geq r_0$; or $\alpha_r \neq 0$ for all sufficiently large~$r$.
\item
\label{enum:iCS-2}
$\alpha_{\bullet} = 0$ if and only if $\rho = 0$ in $C^n(G,M)$ and $\sigma$ lies in $p^{\omega} C^n(G,N)$.
\end{enumerate}
\end{lemma}

\begin{proof}
If $\rho = 0$ and $\sigma$ is not divisible by $p^{\omega}$, then $p^{r-\omega} \sigma$ is not divisible by $p^r$, and so $\alpha_r$ is non-zero for all~$r$.
If $\rho \neq 0$ then there is $k$ such  that $\pro_k(\rho) \neq 0$ in $C^n(G,M/p^k N)$, and hence $\alpha_r \neq 0$ for all $r \geq k+\omega$.
\end{proof}

\begin{notn}
Let $\alpha_{\bullet} \in \cc^n_{r_0}(G,M_{\bullet})$. We define the 
\emph{level} of $\alpha_{\bullet}$ to be the smallest value of $\omega$ 
such that $\alpha_{\bullet}$ is defined by $(\rho,\sigma;r_0,\omega)$ for 
some $\rho,\sigma$.
\end{notn}

\begin{rk}
By definition, the \cohSeq. defined by $(\rho,\sigma;r_0,\omega)$ has level 
at most $\omega$. Note that $(\rho,\sigma;r_0,\omega)$ and 
$(\rho,p\sigma;r_0,\omega+1)$ define the same \cohSeq.. Thus the level of 
the \cohSeq. defined by $(\rho, \sigma; r_0, \omega)$ can be strictly 
smaller than $\omega$.
\end{rk}

\begin{defn}
Define $\cc^n_{r_0}(G,M_{\bullet}) \stackrel{\cbdry}{\rightarrow} 
\cc^{n+1}_{r_0}(G,M_{\bullet})$ by
$(\cbdry \alpha)_r := \cbdry(\alpha_r)$. That is, if 
$\alpha_{\bullet}$ is defined by $(\rho,\sigma;r_0,\omega)$, then 
$\Delta(\alpha_{\bullet})$ is defined by 
$(\Delta(\rho),\Delta(\sigma);r_0,\omega)$. 
Further, write $\zz^n_{r_0}(G,M_{\bullet}) = \ker(\cbdry)$ and 
$\bb^{n+1}_{r_0}(G,M_{\bullet}) = \Bild(\cbdry)$. 
\end{defn}

\noindent
The map $\cbdry$ is $R$-linear and satisfies $\cbdry^2 = 0$. Thus 
$\bb^n_{r_0}(G,M_{\bullet}) \subseteq \zz^n_{r_0}(G,M_{\bullet})$.

\begin{rk}
By Lemma~\ref{lemma:identityCochainSeq} we have $\cbdry(\alpha_{\bullet}) = 0$ if and only if $\cbdry(\rho) = 0$ and $\cbdry(\sigma)$ is divisible by $p^{\omega}$. So we may rephrase Proposition~\ref{prop:splittingCocycleLevel} as follows:
\end{rk}

\begin{corollary}
\label{coroll:new-sCL}
Let $n \geq 1$ and $r_0 \geq 2m$. For every $\bar{\rho} \in H^n(G,M)$ and every $\bar{\eta} \in H^{n+1}(G,N)$ there is a \coySeq. $\alpha_{\bullet} \in \zz^n_{r_0}(G,M_{\bullet})$ of level at most~$m$ such that for every $r \geq r_0$ the cohomology class $\alpha_r + B^n(G,M_r) \in H^n(G,M_r)$ corresponds under the isomorphism of Theorem~\ref{thm18} to $(\bar{\rho},\bar{\eta}) \in H^n(G,M) \oplus H^{n+1}(G,N)$.
\qed
\end{corollary}
 
\begin{lemma}
\label{lemma:coboundarySequence}
Let $n \geq 1$. Suppose  that $\alpha_{\bullet} \in \zz^n_{r_0}(G,M_{\bullet})$ has level $\omega \leq r_0-m$. The following statements are equivalent:
\begin{enumerate}
\item
\label{enum:coboundarySequence-1}
$\alpha_{r_1} \in B^n(G,M_{r_1})$ for some value $r_1 \geq r_0$~of $r$.
\item
\label{enum:coboundarySequence-2}
$\alpha_r \in B^n(G,M_r)$ for every~$r \geq r_0$.
\item
\label{enum:coboundarySequence-3}
$\alpha_{\bullet} \in \bb^n_{r_0}(G,M_{\bullet})$.
\item
\label{enum:coboundarySequence-4}
$\alpha_{\bullet} = \cbdry(\beta_{\bullet})$ for some $\beta_{\bullet}\in \cc^{n-1}_{r_0}(G,M_{\bullet})$ of level at most $\omega + m$.
\end{enumerate}
\end{lemma}

\begin{proof}
The implications \enref{coboundarySequence-4}${} \Rightarrow {}$\enref{coboundarySequence-3}${} \Rightarrow {}$\enref{coboundarySequence-2}${} \Rightarrow {}$\enref{coboundarySequence-1} are clear. \quad
\enref{coboundarySequence-1}${}\Rightarrow{}$\enref{coboundarySequence-4}:
Let $\alpha_{\bullet} = \pro_{\bullet}(\rho + p^{\bullet-\omega} \sigma)$.
Since $\alpha_{r_1} \in B^n(G,M_{r_1})$ there are $\phi \in C^{n-1}(G,M)$ and $\psi \in C^n(G,N)$ such that $\rho + p^{r_1-\omega} \sigma = \cbdry(\phi) + p^{r_1} \psi$, and hence
\[
p^{r_1-\omega} (\sigma - p^{\omega} \psi) = \cbdry(\phi) - \rho \, .
\]
By Lemma~\ref{lemma:identityCochainSeq} we may replace $\sigma$ by $\sigma - p^{\omega} \psi$ without altering $\alpha_{\bullet}$. Hence
\[
p^{r_1-\omega} \sigma = \cbdry(\phi) - \rho \, .
\]
Now, the right hand side is a cocycle, since $\alpha_{\bullet} \in Z^n(G,M_{\bullet})$ means that $\rho$ is a cocycle. Hence the left hand side lies in $Z^n(G,N)$. So $\sigma \in Z^n(G,N)$ by regularity of $p$, and therefore since $p^m H^n(G,N) = 0$ there is $\chi \in C^{n-1}(G,N)$ with $\cbdry(\chi) = p^m \sigma$.
Hence $\rho = \cbdry (\lambda)$ for $\lambda = \phi - p^{r_1-\omega-m} \chi \in C^{n-1}(G,M)$.
So $\alpha_{\bullet} = \cbdry(\beta_{\bullet})$ for $\beta_{\bullet} = \pro_{\bullet}(\lambda + p^{\bullet-\omega-m} \chi)$.
\end{proof}

\noindent
The next result will be needed in the proof of Lemma~\ref{lemma:enoughElabs}\@.

\begin{lemma}
\label{lemma:cohSeqsDense}
Let $n \geq 1$ and $r_1 \geq r_0 \geq 2m$. For each $z \in Z^n(G,M_{r_1})$ there is an $\alpha_{\bullet} \in \zz^n_{r_0}(G,M_{\bullet})$ of level at most $m$ with $\alpha_{r_1} = z$.
\end{lemma}

\begin{proof}
Let $\xi$ be the element of $H^n(G,M) \oplus H^{n+1}(G,N)$ corresponding to $z+B^n(G,M_{r_1}) \in H^n(G,M_{r_1})$ under the isomorphism of Theorem~\ref{thm18}\@. By Corollary~\ref{coroll:new-sCL} there is some $\beta_{\bullet} = \pro_{\bullet}(\rho + p^{\bullet-m} \sigma) \in \zz^n_{r_0}(G,M_{\bullet})$ such that $\beta_r + B^n(G,M_r)$ corresponds to $\xi$ for every $r \geq r_0$. Hence $z - \beta_{r_1} \in B^n(G,M_{r_1})$. Pick $\lambda \in C^{n-1}(G,M_{r_1})$ with $z = \cbdry(\lambda) + \beta_{r_1}$, and choose $\bar{\lambda} \in C^{n-1}(G,M)$ with $\pro_{r_1}(\bar{\lambda}) = \lambda$. For $\rho' = \rho + \cbdry(\bar{\lambda}) \in Z^n(G,M)$ we then have $z = \alpha_{r_1}$ for $\alpha_{\bullet} = \pro_{\bullet}(\rho' + p^{\bullet-m} \sigma)$.
\end{proof}

\section{Coclass families of \texorpdfstring{$p$}{p}-groups}
\label{coclass}

\noindent
Coclass families are certain infinite families of finite $p$-groups of 
fixed coclass. Their construction has been introduced in \cite{ELG08} 
based on a version of Theorem \ref{thm18}. Here we exhibit a construction 
based on Proposition \ref{prop:splittingCocycleLevel}. The construction 
differs from \cite{ELG08} in that it uses cocycles rather than their 
corresponding cocycle classes and thus is slightly more explicit. This 
difference will be essential in our later applications. 

Every coclass family of $p$-groups of coclass $r$ is associated with an 
infinite pro-$p$-group $S$ of coclass $r$. The structure of the infinite 
pro-$p$-groups of finite coclass is well investigated.
For example, it is known that for each such group $S$ there exists natural 
numbers $l$ and $d$ so that the $l$-th term of the lower central series 
$\gamma_l(S)$ satisfies that $\gamma_l(S) \cong \Z_p^d$, where $\Z_p$ 
denotes the $p$-adic numbers, and $S/\gamma_l(S)$ is a finite $p$-group 
of coclass $r$. The integer $d$ is an invariant of $S$ called the 
{\em dimension} of $S$. The integer $l$ is not an invariant; in fact, 
each integer larger than $l$ can be used instead of $l$. The subgroup
$\gamma_l(S)$ is often denoted by $T$ and called a {\em translation 
subgroup} of $S$. Its subgroup series defined by $T_0 = T$ and $T_{i+1} 
= [T_i, S]$ satisfies $[T_i : T_{i+1}] = p$ for $i \in \N_0$. Thus
the series $T = T_0 > T_1 > \ldots$ is the unique series of $S$-normal
subgroups in $T$, and $T$ is called a {\em uniserial} $S$-module. We
refer to \cite{LeedhamGreenMcKay:book} for many more details on the 
structure of the infinite pro-$p$-groups of coclass $r$.

Given $S$ and $l$, we write $S_i = S / \gamma_{l+i}(S)$ for $i \in \N_0$.
Then $S_0, S_1, \ldots$ is an infinite sequence of finite $p$-groups of 
coclass $r$. This sequence is called the {\em main line} associated with 
$S$. The main line is not necessarily a coclass family itself, but it 
always consists of $d$ coclass families and finitely many other groups. 
More precisely, there exists an integer $h \geq l$ so that the $d$ infinite
sequences $(S_{h+i}, S_{h+i+d}, S_{h+i+2d}, \ldots)$ for $0 \leq i < d$ are 
coclass families. Note that the group $S$ can be viewed as an extension of 
$S_{h+i+jd}$ by its natural module $T_{h+i+jd}$ for each $h, i$ and $j$
and the group $S_{h+i+jd+k}$ can be viewed as an extension of $S_{h+i+jd}$
by its natural module $T_{h+i+jd}/T_{h+i+jd+k}$ for each $h,i,j$ and $k$.

For each coclass family $(G_0, G_1, \ldots)$ associated with the 
infinite pro-$p$-group $S$ there exists an integer $k$ so that each group
$G_j$ is a certain extension of $S_{h+i+jd}$ with its natural module 
$T_{h+i+jd} / T_{h+i+jd+k}$. The extensions can be choosen so that 
the main line group $S_{h+i+jd+1}$ is not a quotient of $G_j$. In this case 
the integer $k$ is an invariant of the coclass family called its 
{\em distance} to the main line.

To describe the groups in a coclass family explicitly, it is more convenient
to use a different type of extension construction. Instead of describing
a group $G_j$ in a coclass family as extension of an associated main line 
group $S_{h+i+jd}$ by its natural module of fixed size $p^k$, we describe 
each $G_j$ as an extension of a fixed main line group $S_{\ell}$ for some 
suitable $\ell$ by a module of variable size $M_j := T_{\ell} / T_{\ell+jd}$.
It is not difficult to observe that $T_{\ell + jd} = p^j T_{\ell}$ and thus
the group $M_j$ is isomorphic to a direct product of $d$ copies of cyclic 
groups of order $p^j$.

We now use Proposition~\ref{prop:splittingCocycleLevel} to exhibit a complete
construction for a coclass family $(G_0, G_1, \ldots)$ associated with the
infinite pro-$p$-group $S$ of coclass $r$. For this purpose let 
$m = \log_p(S_\ell) = r+\ell-1$ and let $j \geq 3m +1$. Let $\rho \in 
Z^2(S_\ell, T_\ell)$ so that $S$ is an extension of $S_\ell$ with $T_\ell$
via $\rho$.

\begin{defn}
There exists $\eta \in Z^3(S_\ell, T_\ell)$ so that $G_j$ is an extension
of $S_\ell$ by $M_j$ via $\tau_j$ where 
$\tau_{\bullet} = pro_{\bullet}(\rho + p^{\bullet - m} \sigma)$
and $\cbdry(\sigma) = p^m \eta$.
\end{defn}

The definition of coclass families asserts that for each coclass family 
there exists an $\eta$ yielding this family. It may happen that different
cocycles $\eta_1, \eta_2$ yield coclass families with pairwise isomorphic
groups; for example, this is the case if $\eta_1 \equiv \eta_2 \bmod 
B^3(S_\ell, T_\ell)$. We note that every $\eta \in Z^2(S_\ell, M_j)$ 
yields a coclass family via the above construction.
 
The significance of coclass families is underlined by the fact that for 
$(p,r) = (2,r)$ or $(p,r) = (3,1)$ all but finitely many $p$-groups of 
coclass $r$ are contained in a coclass family. 


\section{\CohSeq.s and elementary abelians}

\noindent
We now apply the results of Section~\ref{CohSeqs} to the coclass family $G_{\bullet}$ of Section~\ref{coclass}\@.  In the language of Notation~\ref{setup2} this means that $M = T$. It would be natural to take $N=T$ as well, but for technical reasons\footnote{It simplifies Remark~\ref{rk:firstW} and especially Lemma~\ref{lemma:intermediate} if $M/p^r N$ and $M/N$ have the same elementary abelian subgroups.} we shall actually take $N = pT$. Hence $M_{\bullet} = M/p^{\bullet}N = T/p^{\bullet+1} T$, and $G_{\bullet+1} = M_{\bullet} . P$ with extension cocycle $\tau_{\bullet} \in \zz^2_{r_0}(P,M_{\bullet})$\@.
\medskip

\noindent
Recall that if $N \trianglelefteq G$ and $U \leq G/N$, then a \emph{lift} of $U$ is a subgroup $\bar{U} \leq G$ such that the projection map $G \rightarrow G/N$ maps $\bar{U}$ isomorphically to~$U$.

Suppose that we are given $H \leq G_{r+1}$. Setting $K := H \cap M_r$ and $Q := HM_r/M_r$, we see that $H$ is an extension $H = K .Q$, with $Q \leq P$, and $K$ a $Q$-submodule of~$M_r$.  If $K$ has a complement $C$~in $H$ -- which is certainly the case if $H$ is elementary abelian --, then $C \leq G_{r+1}$ is a lift of $Q$.


\subsection{Extension theory}
We recall some details of extension theory, see e.g.\@ \cite[\S3.7]{Benson:I}\@.
Let $G$ be a finite group and $M$ a left $\zz G$-module. Recall that every group extension $\Gamma = M.G$ can be constructed using a 2-cocycle $\tau \in Z^2(G,M)$: the underlying set is $M \times G$, with multiplication
\[
(t_1,g_1)(t_2,g_2) = (t_1 + {}^{g_1}t_2 + \tau(g_1,g_2), g_1g_2) \, .
\]
Associativity is equivalent to the cocycle condition. The extension splits as a semidirect product $\Gamma = M \rtimes G$ if and only $\tau \in B^2(G,M)$. If $\tau = \cbdry(f)$ then $G(f) = \left\{ \left(-f(g), g\right) \mid g \in G \right\} \leq \Gamma$ is a lift of $G$, and every lift arises thus.

\begin{lemma}
\label{lemma:preUsplit}
If $f,f' \in C^1(G,M)$ satisfy $\cbdry(f) = \tau =  \cbdry(f')$ then $f'-f \in Z^1(G,M)$ and moreover
\[
f'-f \in B^1(G,M) \; \Longleftrightarrow \; \text{$G(f)$ and $G(f')$ are conjugate by an element of~$M$.}
\]
\end{lemma}

\begin{proof}
This is Proposition 3.7.2 of~\cite{Benson:I}\@. Observe from the proof of that result that conjugation by elements of~$M$ induces every coboundary.
\end{proof}

\subsection{Lifting elementary abelians}

\begin{lemma}
\label{lemma:Usplit}
Let $Q \leq P$ and suppose $r_0 \geq 2m$.
Then the three following statements are equivalent:
\begin{enumerate}
\item
\label{enum:Usplit-1}
$Q$ has has a lift $\bar{Q}_r \leq G_{r+1}$ for all~$r \geq r_0$.
\item
\label{enum:Usplit-2}
$Q$ has has a lift $\bar{Q}_r \leq G_{r+1}$ for at least one~$r \geq r_0$.
\item
\label{enum:Usplit-3}
$\tau_{\bullet} \vert_Q = \cbdry(f_{\bullet})$ for some \cohSeq. $f_{\bullet} \in \cc^1_{r_0}(Q,M_{\bullet})$ of level at most~$m$.
\end{enumerate}
\end{lemma}

\begin{proof}
Let $H_r \leq G_{r+1}$ be the subgroup with $M_r \leq H_r$ and $H_r/M_r = Q$. Then $H_r = M_r . Q$ with extension class $\tau_r \vert_Q$, and $Q$ has a lift $\bar{Q}_r \leq G_{r+1}$ if and only if $\tau_r\vert_Q$ lies in $B^2(Q,M_r)$. Now apply Lemma~\ref{lemma:coboundarySequence}\@.
\end{proof}

\noindent
Now suppose that $E \leq G_{r+1}$ is elementary abelian. Setting $K = E \cap M_r$ and $U = EM_r / M_r \leq P$ as above, we have $E = K \times \bar{U}$ for a lift $\bar{U} \leq G_{r+1}$ of~$U$. Recall that $\bar{U} = U(\phi)$ for some $\phi \in C^1(U,M_r)$ with $\cbdry(\phi) = \tau_r\vert_U$.

\begin{notn}
We shall need to refer to several different maps between cohomology modules. Let $L \subseteq M$ be a submodule.
\begin{itemize}
\item
Inclusion $L \hookrightarrow M$ induces $H^n(G,L) \xrightarrow{\inc} H^n(G,M)$.
\item
$\mul^r \colon M/L \rightarrow M/p^r L$, $x + L \mapsto p^r x + p^r L$ induces $H^n(G,M/L) \xrightarrow{\mul^r} H^n(G,M/p^r L)$. 
\end{itemize}
Note that $\mul^r \circ \mul^s = \mul^{r+s}$, and $\pro_r(\rho + p^{r-m} \sigma) = \pro_r(\rho) + \mul^{r-m} \pro_m(\sigma)$.
\end{notn}

\begin{remark}
\label{rk:firstW}
Since $K \leq M_r = T/p^{r+1}T$ is elementary abelian, it follows that
\[
K \leq \Omega_1(M_r) = p^rT/p^{r+1}T \xrightarrow[\cong]{(\mul^r)^{-1}} T/pT \, .
\]
So $K = \mul^r(W)$ for some $W \leq T/pT$. Since $E$ is abelian, we have $[\bar{U},K]=1$, which is equivalent to $W \leq (T/pT)^U$.
\end{remark}

\begin{notn}
$\mathcal{E}$ is the set of all triples $(U,f_{\bullet},W)$ with $U \leq P$ an elementary abelian; $f_{\bullet} \in \cc^1_{r_0}(U,M_{\bullet})$ a \cohSeq. of level at most $2m$ such that $\cbdry(f_{\bullet}) = \tau_{\bullet}\vert_U$; and $W \leq (T/pT)^U$.
\end{notn}

\begin{lemma}
\label{lemma:enoughElabs}
Suppose that $r \geq r_0 \geq 2m$. Every elementary abelian $E \leq G_{r+1}$ has the form $E = E_r(U,f_{\bullet},W) := \mul^r(W) \times U(f_r)$ for some $(U,f_{\bullet},W) \in \mathcal{E}$.
\end{lemma}

\begin{proof}
We saw above that $E = \mul^r(W) \times U(\phi)$ with $U \leq P$ elementary abelian; $W \leq (T/pT)^U$; and $\phi \in C^1(U,M_r)$ with $\cbdry(\phi) = \tau_r \vert_U$. As $U(\phi)$ is a lift of $U$ in $G_{r+1}$, there is $f_{\bullet} \in \cc^1_{r_0}(U,M_{\bullet})$ of level at most $2m$ with $\cbdry(f_{\bullet}) = \tau_{\bullet}\vert_U$ by  Lemma~\ref{lemma:Usplit}\@. Hence $\phi - f_r \in Z^1(U,M_r)$, so by Lemma~\ref{lemma:cohSeqsDense} there is $z_{\bullet} \in \zz^1_{r_0}(U,M_{\bullet})$ of level at most $m$ with $z_r = \phi - f_r$. So $(U, f_{\bullet} + z_{\bullet}, W) \in \mathcal{E}$, and $E = \mul^r(W) \times U(f_r+z_r)$.
\end{proof}

\section{Change of module}
\label{section:change}
\noindent
The following technical lemma is required in the proofs of Proposition~\ref{prop:maps} and Lemma~\ref{lemma:subgroups}\@. We revert to Notation~\ref{setup2}\@, and consider the case of two submodules $L,N \subseteq M$ satisfying the condition $pL \subseteq N \subseteq L$.

\begin{bsp}
If $(U,f_{\bullet},W) \in \mathcal{E}$, then $W \leq T/pT$, and so $W = L/pT$ for some $pT \subseteq L \subseteq T$. Hence $pL \subseteq N \subseteq L$, since $N = pT$.
\end{bsp}

\noindent
We shall investigate the cohomology maps induced by the short exact sequence
\[
0 \rightarrow L/N \xrightarrow{\mul^r} M/p^rN \rightarrow M/p^rL \rightarrow 0 \, .
\]
As we now have to distinguish between two different projection maps, we shall denote them by $M \xrightarrow{\pro^N_r} M/p^r N$ and $M \xrightarrow{\pro^L_r} M/p^r L$.

\begin{lemma}
\label{lemma:intermediate}
Suppose the $RG$-submodule $L \subseteq M$ satisfies $pL \subseteq N \subseteq L$.
\begin{enumerate}
\item
\label{enum:intermediate-1}
Assume $r_0 \geq 1$.
Then $j_* \circ i_* = 0$ for the chain maps
\[
C^*(G,L/N) \stackrel{i_*}{\longrightarrow} \cc^*_{r_0}(G,M/p^{\bullet} N) \stackrel{j_*}{\longrightarrow} \cc^*_{r_0}(G,M/p^{\bullet} L)
\]
given by $i_n(c)_r = \mul^r(c)$ and $j_n(\alpha_{\bullet})_r = \alpha_r + p^r L$.
\item
\label{enum:intermediate-2}
Suppose that $\alpha_{\bullet} \in \zz^n_{r_0}(G,M/p^{\bullet} N)$ satisfies $j_n(\alpha_{\bullet}) \in \bb^n_{r_0}(G,M/p^{\bullet} L)$.
If $\alpha_{\bullet}$ has level $\omega \leq r_0-m$, then
\[
\alpha_{\bullet} = i_n(c)_{\bullet} + \cbdry(\pro^N_{\bullet}(\kappa + p^{\bullet -(\omega+m)} \lambda))
\]
for some $c \in Z^n(G,L/N)$, $\kappa \in C^{n-1}(G,M)$ and $\lambda \in C^{n-1}(G,L)$.
\end{enumerate}
\end{lemma}

\begin{proof}
\enref{intermediate-1}:
Pick $\bar{c} \in C^n(G,L)$ with $\pro^L_0(\bar{c}) = c$, then $p\bar{c} \in C^n(G,N)$ and
\[
i_n(c)_{\bullet} = \pro^N_{\bullet}(0 + p^{\bullet-1} \cdot p\bar{c}) \in \cc^n_{r_0}(G,M/p^{\bullet} N) \, , \quad \text{with level $1 \leq r_0$.}
\]
If $\alpha_{\bullet} = \pro^N_{\bullet}(\rho + p^{\bullet - \omega} \sigma) \in \cc^n_{r_0}(G,M/p^{\bullet} N)$,
then $j_n(\alpha_{\bullet}) = \pro^L_{\bullet}(\rho + p^{\bullet-\omega} \sigma)$.
\\
Clearly $i_*,j_*$ are chain maps. And $j_n i_n = 0$, since $c$ takes values in $L$.

\enref{intermediate-2}:
Let $\alpha_{\bullet} = \pro^N_{\bullet}(\rho + p^{\bullet - \omega} \sigma)$, so $j_n(\alpha_{\bullet}) = \pro^L_{\bullet}(\rho + p^{\bullet-\omega} \sigma)$. By Lemma~\ref{lemma:coboundarySequence} we have $j_n(\alpha_{\bullet}) = \cbdry(\gamma_{\bullet})$ for some $\gamma_{\bullet} \in \cc^{n-1}_{r_0}(G,M/p^{\bullet} L)$ of the form
\[
\gamma_{\bullet} = \pro^L_{\bullet}(\kappa + p^{\bullet-\omega-m} \lambda) \quad
\text{with $\kappa \in C^{n-1}(G,M)$, $\lambda \in C^{n-1}(G,L)$.}
\]
Applying Lemma~\ref{lemma:identityCochainSeq} we have $\rho = \cbdry(\kappa)$, and
\[
p^m \sigma = \cbdry(\lambda)  + p^{\omega + m} \bar{c} \quad \text{for some $\bar{c} \in C^n(G,L)$.}
\]
From $\alpha_{\bullet} \in \zz^n_{r_0}(G,M/p^{\bullet} N)$ it follows that $\cbdry(\sigma)$ takes values in $p^{\omega}N$, and so $\cbdry(\bar{c})$ takes values in $N$. 
So $c := \pro^N_0(\bar{c})$ lies in $Z^n(G, L/N)$, and $i_n(c)_r = \pro^N_r(p^r \bar{c})$. Hence
\begin{align*}
\alpha_{\bullet} & = \pro^N_{\bullet}(\rho + p^{\bullet - \omega} \sigma)
\\ & = \pro^N_{\bullet}(\cbdry(\kappa + p^{\bullet - (\omega + m)} \lambda) + p^{\bullet} \bar{c})
\\ & = i_n(c) + \cbdry (\pro_{\bullet}(\kappa + p^{\bullet - (\omega + m)} \lambda)) \, .
\qedhere
\end{align*}
\end{proof}

\ignore{
\begin{notn}
Write $L_{\omega} \cc^n_{r_0}(G,M_{\bullet})$ for the submodule of \cohSeq.s of level at most~$\omega$.
$L_{\omega} \cc^*_{r_0}(G,M_{\bullet})$ is a chain complex, since $\cbdry(L_{\omega} \cc^n_{r_0}(G,M_{\bullet})) \subseteq L_{\omega} \cc^{n+1}_{r_0}(G,M_{\bullet})$.
\end{notn}
}

\section{Morphisms in the Quillen category}

\begin{notn}
Consider the triple $(U,f_{\bullet},W) \in \mathcal{E}$. Recall from Section~\ref{prelim-cohom} that $G_{r+1}$ has underlying set $M_r \times P$, and that $U(f_r) = \{(-f_r(u),u) \mid u \in U \}$. So the subgroup $E_r(U,f_{\bullet},W) \leq G_{r+1}$ of Lemma~\ref{lemma:enoughElabs} is
\[
E_r(U,f_{\bullet},W) =  \{(p^r w - f_r(u), u) \mid u \in U, w \in W \} \, .
\]
Let $j_r^f \colon W \times U \rightarrow E_r(U,f_{\bullet},W)$ be the isomorphism $(w,u) \mapsto (p^r w - f_r(u), u)$.
\end{notn}

\begin{proposition}
\label{prop:maps}
Suppose $r_0 \geq 3m$ and $m \geq 1$.
For $(U,f_{\bullet},W), (U',f'_{\bullet},W') \in \mathcal{E}$ the set of isomorphisms $W \times U \rightarrow W' \times U'$ of the form
\[
\begin{CD}
W \times U @>{j_r^f}>> E_r(U,f_{\bullet},W) @>{\text{conjugation in $G_{r+1}$}}>> E_r(U',f'_{\bullet},W') @>{(j_r^{f'})^{-1}}>> W' \times U'
\end{CD}
\]
is independent of~$r$.
\end{proposition}

\noindent
For the proof we need two lemmas. Observe that $\mul^{r-r_0}$ embeds $M_{r_0} = T/p^{r_0+1} T$ in~$G_{r+1}$ as $p^{r-r_0} T/p^{r+1}T \leq M_r$.

\begin{lemma}
\label{lemma:AutT}
Suppose that $m \geq 1$ and $r_0 \geq 2m$.
Let $(U,f_{\bullet},W) \in \mathcal{E}$. 
\begin{enumerate}
\item
\label{enum:AutT-3}
$\Aut_{M_r}(E_r(U,f_{\bullet},W)) = \Aut_{\mul^{r-r_0}(M_{r_0})}(E_r(U,f_{\bullet},W))$.
\item
\label{enum:AutT-4}
The subgroup
\[
\mathcal{N} = \{ x \in M_{r_0} \mid \mul^{r-r_0}(x) \in N_{G_{r+1}}(E_r(U,f_{\bullet},W)) \}
\]
depends on neither $r$~nor $f_{\bullet}$. Nor does the action of $\mathcal{N}$ on $W \times U$ obtained by using $j^f_r$ to identify $E_r(U,f_{\bullet},W)$ with $W \times U$.
\end{enumerate}
\end{lemma}

\noindent
Write $\bar{W}$ for the module $pT \subseteq \bar{W} \subseteq T$ with $W = \bar{W}/pT$.

\begin{proof}
\enref{AutT-3}:
Conjugation by $t \in T$ fixes $M_r U(f_r)/M_r$ and $W$ pointwise, and is described by $\cbdry(t) \in B^1(U,M_r)$:
\[
{}^{(t,0)}(p^r w - f_r(u), u) = (p^r w - \cbdry(t)(u) - f_r(u), u) \, .
\]
If $t$ normalizes $\mul^r(W) \times U(f_r) \leq G_{r+1}$ then $\cbdry(t)$ must take values in $p^r \bar{W}$. Hence $\cbdry(t) \in p^r Z^1(U,\bar{W}) \subseteq p^{r-m} B^1(U,\bar{W})$. So there is $\bar{v} \in \bar{W}$ such that $\cbdry(t) = p^{r-m} \cbdry (\bar{v})$, and $\pro_r(p^{r-m} \bar{v}) = \mul^{r-r_0} \pro_{r_0}(p^{r_0-m} \bar{v}) \in \mul^{r-r_0}(M_{r_0})$ has the same conjugation action as $t$.

\enref{AutT-4}: Conversely if $v = \bar{v} + p^{r_0+1} T \in M_{r_0}$ then $\mul^{r-r_0} (v)$ normalizes $\mul^r(W) \times U(f_r)$ if and only if $\cbdry(p^{r-r_0} \bar{v}) \in Z^1(U,T)$ takes values in $p^r \bar{W}$, that is if $\cbdry(v) = \mul^{r_0}(z)$ for some $z \in Z^1(U,W)$. And the action on $W \times U$ is then $(w,u) \mapsto (w - z(u),u)$.
\end{proof}

\begin{lemma}
\label{lemma:chi}
Suppose $r_0 \geq 2m$ and $g \in P$.
Let $(U,f_{\bullet},W), ({}^gU,f'_{\bullet},{}^gW) \in \mathcal{E}$. Define $\chi_r \in C^1({}^gU,M_r)$ by
\[
\chi_r(v) = ({}^gf_r)(v) - \tau_r(g,v^g) + \tau_r(v,g) - f'_r(v) \, .
\]
Then $\chi_{\bullet} \in \zz^1_{r_0}({}^gU,M_{\bullet})$ is a \coySeq. of level at most $2m$.
\end{lemma}

\begin{proof}
For $c \in C^n(H,M)$ we of course define ${}^gc \in C^n({}^gH,M)$ by
\[
({}^gc)(k_1,\ldots,k_n) = {}^g c(k_1^g,\ldots,k_n^g) \, .
\]
Everything is of level at most $2m$. Since $\cbdry(f'_{\bullet}) = \tau_{\bullet}\vert_{{}^gU}$ and $\cbdry({}^gf_{\bullet}) = {}^g (\tau_{\bullet}\vert_U)$ we have
\begin{align*}
\cbdry(\chi_r)(v_1,v_2) & = {}^g \tau_r(v_1^g, v_2^g) - \tau_r(g,v_2^g) + \tau_r(g,v_1^g v_2^g) - \tau_r(g,v_1^g) \\
& \qquad {} + \tau_r(v_2,g) - \tau_r(v_1v_2,g) + \tau_r(v_1,g) - \tau_r(v_1,v_2) \\
& =  \cbdry(\tau_r)(g,v_1^g,v_2^g) - \cbdry (\tau_r)(v_1,g,v_2^g) + \cbdry(\tau_r)(v_1,v_2,g) = 0 \, .
\qedhere
\end{align*}
\end{proof}

\begin{proof}[Proof of Proposition~\ref{prop:maps}]
If ${}^{(t,g)}E_r(U,f_{\bullet},W) = E_r(U',f'_{\bullet},W')$ then ${}^gU = U'$ and ${}^gW = W'$. So we assume that $g \in P$ is fixed, and consider which $t \in M_r$ satisfy ${}^{(t,g)}E_r(U,f_{\bullet},W) = E_r({}^gU,f'_{\bullet},{}^gW)$, and which isomorphisms arise in this way.

The map $F \colon W \times U \rightarrow {}^g W \times {}^gU$ given by $j_r^f$, then conjugation by $(t,g)$, and then $(j_r^{f'})^{-1}$ must have the form $F(w,u) = ({}^gw - \pi({}^gu), {}^g u)$ for some $\pi \in Z^1({}^gU,{}^gW)$. So we consider $g,\pi$ to be fixed and ask for which values of~$r$ there is $t + p^{r+1}T \in M_r$ realising this~$F$.
The condition on $t,i$ can be phrased thus:
\[
(t,g) (p^r w - f_r(u), u) = (p^r \cdot {}^gw - p^r \pi({}^gu) - f'_r({}^gu), {}^gu) (t,g)
\]
Equality in ${}^g U$ is immediate. We are left with the following condition in~$M_r$:
\[
t + p^r\cdot {}^g w - {}^g(f_r(u)) + \tau_r(g,u) = p^r \left({}^gw - \pi({}^gu)\right) - f'_r({}^gu) + {}^{({}^gu)} t + \tau_r({}^gu, g) \, .
\]
So with $\chi_r$ as in Lemma~\ref{lemma:chi} we have
\begin{align*}
p^r \pi({}^g u) & = ({}^g f)({}^g u) - \tau_r(g,u) + \tau_r({}^g u,g) - f'_r({}^g u) + \cbdry(t)({}^g u) \\ & = (\chi_r + \cbdry(t))({}^g u) \, .
\end{align*}
That is, $F$ is realisable for this~$r$ if and only if
\begin{equation}
\label{eqn:pi}
p^r \pi - \chi_r \in B^1({}^g U, M_r) \, .
\end{equation}
Since $\pi$ takes values in ${}^g W = {}^g \bar{W} / pT$, a necessary condition for any such $F$ to be realisable is that
\begin{equation}
\label{eqn:necCondn}
\chi_r + p^r \cdot {}^g \bar{W} \in B^1({}^g U, T/p^r \cdot {}^g \bar{W}) \, .
\end{equation}
Since $\chi_{\bullet} \in \zz^1_{r_0}({}^gU,M_{\bullet})$ has level at most $2m$ and $r_0 \geq 2m+m$, we deduce from Lemma~\ref{lemma:coboundarySequence} that~\eqref{eqn:necCondn} is either satisfied for all~$r_0$, or for none.

If~\eqref{eqn:necCondn} is satisfied then we apply Lemma~\ref{lemma:intermediate} with $G = {}^gU$, $\alpha_{\bullet} = \chi_{\bullet}$ and $L = {}^g\bar{W}$, hence $L/N = {}^g W$. Note that $\chi_{\bullet}$ has level at most $2m \leq r_0-m$. We conclude that there are $c \in Z^1({}^gU,{}^gW)$, $\kappa \in C^0({}^gU, T)$ and $\lambda \in C^0({}^gU,{}^g\bar{W})$ with $\chi_{\bullet} = \mul^{\bullet} (c) + \cbdry(\pro_{\bullet}(\kappa + p^{\bullet-3m}\lambda))$.
We conclude that if we take $\pi = c$ then Eqn.~\eqref{eqn:pi} is solvable for all~$r$. That is, this one map $F \colon W \times U \rightarrow {}^g W \times {}^gU$ is independent of~$r$. But all other maps for this value of~$g$ correspond to a $M_r$-automorphism of $U \times W$ followed by~$F$, and we saw in Lemma~\ref{lemma:AutT} that these isomorphisms are independent of~$r$ too. 
\end{proof}

\begin{corollary}
\label{coroll:conjugacySequence}
Suppose $e \geq 3m$ and $m \geq 1$.
For $(U,f_{\bullet},W), (U',f'_{\bullet},W') \in \mathcal{E}$ the following statements are equivalent:
\begin{enumerate}
\item
$E_r(U,f_{\bullet},W)$ and $E_r(U',f'_{\bullet},W')$ are $G_{r+1}$-conjugate for some~$r$.
\item
$E_r(U,f_{\bullet},W)$ and $E_r(U',f'_{\bullet},W')$ are $G_{r+1}$-conjugate for every~$r$.
\end{enumerate}
\end{corollary}

\begin{proof}
They are $G_{r+1}$-conjugate if and only if the set of isomorphisms in Proposition~\ref{prop:maps} is non-empty. But this set does not depend on~$r$.
\end{proof}

\section{Wrapping up the main theorem}

\begin{lemma}
\label{lemma:subgroups}
Suppose $r_0 \geq 2m$.
Let $(U,f_{\bullet},W) \in \mathcal{E}$. For each $V \leq W \times U$ there is $(U',f'_{\bullet},W') \in \mathcal{E}$ such that $\forall \, r \::\:  j^f_r(V) = E_r(U',f'_{\bullet},W')$\@. Moreover the map
\[
\kappa^V \colon W' \times U' \xrightarrow{j^{f'}_r} E_r(U',f'_{\bullet},W') \hookrightarrow E_r(U,f_{\bullet},W) \xrightarrow{\left(j^{f}_r\right)^{-1}} W \times U
\]
has image $V$ and is independent of~$r$.
\end{lemma}

\begin{proof}
Take $W' = V \cap W$ and $U' = \{u \in U \mid uW \in VW/W\} \leq U$. Then $V/W' \cong U'$ and $W'$ is a direct factor of~$V$, so there is $c \in Z^1(U',W)$ with $V = \{(w-c(u),u) \mid w \in W', u \in U' \}$. Then
\[
j^f_r(V) = \{ (p^{i+e-1}w - p^{i+e-1}c(u)-f_r(u),u) \mid w \in W', u \in U' \} \, .
\]
Done with $f'_{\bullet} = f_{\bullet}\vert_{U'} + i_1(c)_{\bullet}$ in the terminology of Lemma~\ref{lemma:intermediate}\@. In particular, $\kappa^V(w,u) = (w-c(u),u)$.
\end{proof}

\begin{corollary}
\label{coroll:injMaps}
Suppose $r_0 \geq 3m$ and $m \geq 1$.
For $(U,f_{\bullet},W), (U',f'_{\bullet},W') \in \mathcal{E}$ the set of homomorphisms $W \times U \rightarrow W' \times U'$ of the form
\[
\begin{CD}
W \times U @>{j_r^f}>> E_r(U,f_{\bullet},W) @>{\text{morphism in $\Aa_p(G_{r+1})$}}>> E_r(U',f'_{\bullet},W') @>{(j_r^{f'})^{-1}}>> W' \times U'
\end{CD}
\]
is independent of~$r$.
\end{corollary}

\begin{proof}
Every such map is an isomorphism to some $V \leq W' \times U'$. Lemma~\ref{lemma:subgroups}: For some $(U'',f''_{\bullet},W'')$ have $j^{f'}_r(V) = E_r(U'',f''_{\bullet},W'')$ for all~$r$.
Proposition~\ref{prop:maps}: The set $\mathcal{I}_V$ of isomorphisms of the form
\[
\begin{CD}
W \times U @>{j_r^f}>> E_r(U,f_{\bullet},W) @>{\text{conjugation in $G_{r+1}$}}>> E_r(U'',f''_{\bullet},W'') @>{(j_r^{f''})^{-1}}>> W'' \times U''
\end{CD}
\]
is independent of~$r$. But $\phi \mapsto \kappa^V \circ \phi$ is a bijection from $\mathcal{I}_V$ to the set of homomorphisms of the form
\[
\begin{CD}
W \times U @>{j_r^f}>> E_r(U,f_{\bullet},W) @>{\text{morphism in $\Aa_p(G_{r+1})$}}>> E_r(U',f'_{\bullet},W') @>{(j_r^{f'})^{-1}}>> W' \times U'
\end{CD}
\]
whose image is~$V$.
\end{proof}

\begin{proposition}
\label{prop:main}
Suppose $e \geq 3m$ and $m \geq 1$. Choose a subset $\mathcal{E}_0 \subseteq \mathcal{E}$ such that for every conjugacy class $C$ of elementary abelian subgroups in $G_{r_0+1}$ there is exactly one $(U,f_{\bullet},W) \in \mathcal{E}_0$ such that $E_{r_0}(U,f_{\bullet},W)$ lies in~$C$. Define $\mathcal{C}_r$ to be the full subcategory of the Quillen category $\Aa_p(G_{r+1})$ on the $E_r(U,f_{\bullet},W)$ with $(U,f_{\bullet},W)$ in $\mathcal{E}_0$. Then:
\begin{enumerate}
\item
\label{enum:propMain-1}
$\mathcal{C}_r$ is a skeleton of~$\Aa_p(G_{r+1})$ for every~$r \geq r_0$.
\item
\label{enum:propMain-2}
The categories $\mathcal{C}_r$ are all isomorphic to each other.
\end{enumerate}
Hence the Quillen categories $\Aa_p(G_{r+1})$ are all equivalent to each other.
\end{proposition}

\begin{proof}
$\mathcal{E}_0$ exists by Lemma~\ref{lemma:enoughElabs}\@.
\quad \enref{propMain-1}: Need to show that each conjugacy class $C$ in $G_{r+1}$ contains $E_r(U,f_{\bullet},W)$ for precisely one $(U,f_{\bullet},W) \in \mathcal{E}_0$. Corollary~\ref{coroll:conjugacySequence}: at most one such triple. Lemma~\ref{lemma:enoughElabs}: there is some $(U',f'_{\bullet},W') \in \mathcal{E}$ such that $E_r(U',f'_{\bullet},W')$ lies in~$C$. By construction of $\mathcal{E}_0$ there is $(U,f_{\bullet},W) \in \mathcal{E}_0$ such that $E_0(U,f_{\bullet},W)$, $E_0(U',f'_{\bullet},W')$ are $G_0$-conjugate. So $E_r(U,f_{\bullet},W)$ lies in~$C$ by Corollary~\ref{coroll:conjugacySequence}\@.

\enref{propMain-2}: For $r,r' \geq r_0$ and $(U,f_{\bullet},W) \in \mathcal{E}_0$ have isomorphism
\[
\lambda^f_{rr'} \colon E_r(U,f_{\bullet},W) \xrightarrow{\left(j^f_r\right)^{-1}} W \times U \xrightarrow{j^f_{r'}} E_{r'}(U,f_{\bullet},W) \, ,
\]
mit $\lambda^f_{r'r} = \left(\lambda^f_{rr'}\right)^{-1}$. For a morphism $E_r(U,f_{\bullet},W) \xrightarrow{\phi} E_r(U',f'_{\bullet},W')$ in $\mathcal{C}_r$, define $F(\phi)$ in $\mathcal{C}_{r'}$ thus:
\[
F(\phi) \colon E_{r'}(U,f_{\bullet},W) \xrightarrow{\lambda^f_{r'r}} E_r(U,f_{\bullet},W) \xrightarrow{\phi} E_r(U',f'_{\bullet},W') \xrightarrow{\lambda^{f'}_{rr'}} E_{r'}(U',f'_{\bullet},W') \, . 
\]
This is a bijection
\[
\mathcal{C}_r(E_r(U,f_{\bullet},W), E_r(U',f'_{\bullet},W')) \rightarrow \mathcal{C}_{r'}(E_{r'}(U,f_{\bullet},W), E_{r'}(U',f'_{\bullet},W'))
\]
by Corollary~\ref{coroll:injMaps}, and it is functorial since $F(\Id_{W \times U_r(f)}) = \lambda^f_{rr'} \Id \lambda^f_{r'r} = \Id_{W \times U_{r'}(f)}$, and for $E_r(U',f'_{\bullet},W') \xrightarrow{\psi} E_r(U'',f''_{\bullet},W'')$ in $\mathcal{C}_r$ have
\[
F(\psi) F(\phi) = \lambda^{f''}_{rr'} \psi \lambda^{f'}_{r'r} \circ \lambda^{f'}_{ir'} \phi \lambda^f_{r'r} = \lambda^{f''}_{rr'} \psi \phi \lambda^f_{r'r} = F(\psi \phi) \, .
\]
This establishes~\enref{propMain-2}\@. The last part follows from \enref{propMain-1}~and \enref{propMain-2}\@.
\end{proof}

\section{Examples}
\label{section:Examples}

\subsection{Main line maximal class groups}
\label{subsection:Gp}

\noindent
If $p$ is an odd prime then $\mathcal{G}(p,1)$ consists of one infinite tree, together with the isolated point $C_{p^2}$: so there is only one uniserial $p$-adic space group of coclass one. We recall the construction of the main line groups from Example 3.1.5(ii) of~\cite{LeedhamGreenMcKay:book}\@.

The $p$th local cyclotomic number field is $K = \qq_p(\theta)$, where $\theta$ has minimal polynomial $\Phi_p(X) = \frac{X^p-1}{X-1}$. The ring of integers in~$K$ is $\mathcal{O} = \zz_p[\theta]$, a free $\zz_p$-module of rank $p-1$ with basis $1,\theta,\ldots,\theta^{p-2}$. The coclass one uniserial $p$-adic space group is then $G := \mathcal{O} \rtimes C_p$, where the generator $\tau$~of $C_p$ acts as multiplication by~$\theta$; that is, ${}^{\tau} v = \theta v$ for $v \in \mathcal{O}$.

The valuation ring $\mathcal{O}$ has unique maximal ideal $\alpha \mathcal{O}$, where $\alpha = \theta - 1$. So $\gamma_i(G) = \alpha^{i-1} \mathcal{O}$ for $i \geq 2$; and by considering $\Phi_p(X+1)$ one observes that $p\mathcal{O} = \alpha^{p-1} \mathcal{O}$.
Since $1,\alpha,\ldots,\alpha^{p-2}$ is a $\zz_p$-basis of~$\mathcal{O}$, this means that $\mathcal{O}/\alpha\mathcal{O} \cong \f$, and hence $\gamma_i(G)/\gamma_{i+1}(G) \cong C_p$ for $i \geq 2$.

The main line groups are the quotients $G_i = G/\gamma_i(G)$. These main line groups fall into $p-1$ coclass families, where for $0 \leq r \leq p-2$ the $i$th group in the $r$th family is $G_{r+(p-1)i}$. From~\cite{quillen} (see also Proposition~\ref{prop:main}) we know that all groups in one coclass family have equivalent Quillen categories. But here a stronger result holds: all $p-1$ coclass families have the same equivalence class of Quillen categories.

\begin{lemma}
For this group $G = \mathcal{O} \rtimes C_p$, the Quillen category of $G/\gamma_i(G)$ is independent (up to equivalence of categories) of~$i$ for $i \geq p+1$.
\end{lemma}

\begin{rk}
For $p=3$, the first author and S.~King~\cite{EickKing:IsoProblem} have shown that $G/\gamma_5(G)$, $G/\gamma_6(G)$ and $G/\gamma_7(G)$ have isomorphic cohomology rings; and that these differ from the cohomology ring of~$G/\gamma_4(G) \cong 3^{1+2}_+$.
\end{rk}

\begin{proof}
If $v \in \mathcal{O}$ then $(v\tau)^p = (\Phi_{p-1}(\theta) \cdot v)\tau^p = 1$, and so $v\tau$ has order~$p$. So since $p\mathcal{O} = \alpha^{p-1} \mathcal{O}$, there are two kinds of order~$p$ elements of $G/\gamma_i(G)$:
\begin{itemize}
\item Elements of the form $v\tau^r \gamma_i(G)$, with $v \in \mathcal{O}$ and $1 \leq r \leq p-1$; and
\item Elements of $\gamma_{i-p+1}(G)/\gamma_i(G) \cong (C_p)^{p-1}$.
\end{itemize}
Moreover the conjugacy class of $v\tau$ in $G$ is $\{ w\tau \mid w \in v + \alpha \mathcal{O} \}$; and the centralizer of $v\tau \gamma_i(G)$ in $G/\gamma_i(G)$ is elementary abelian of rank~$2$, generated by $v\tau \gamma_i(G)$ and $\gamma_{i-1}(G)/\gamma_i(G)$. So as $\tau$ acts on $\gamma_{i-p+1}(G)/\gamma_i(G) = \alpha^{i-p} \mathcal{O} / \alpha^{i-1} \mathcal{O}$ as multiplication by $1+ \alpha$, the objects of the Quillen category form the following equivalence classes:
\begin{itemize}
\item The class of $\langle v_j \tau\gamma_i(G), \gamma_{i-1}(G)/\gamma_i(G) \rangle \cong C_p^2$ for some fixed transversal $v_1,\ldots,v_p$ of $\mathcal{O}/\alpha \mathcal{O}$;
\item The class of $\langle v_j \tau\gamma_i(G) \rangle \cong C_p$ for the same transversal $v_1,\ldots,v_p$; and
\item The conjugacy classes of subgroups of $\gamma_{i-p+1}(G)/\gamma_i(G) \cong \mathcal{O}/\alpha^{p-1} \mathcal{O} \cong C_p^{p-1}$ under the action of $C_p$ given by multiplication by $1 + \alpha$.
\end{itemize}
So the equivalence classes of objects admit a description which is independent of~$i$. From this description and the fact that $\langle v_j \tau \gamma_i(G), \gamma_{i-1}(G)/\gamma_i(G) \rangle$ has normalizer $\langle v_j \tau \gamma_i(G), \gamma_{i-2}(G)/\gamma_i(G) \rangle$, it follows that the morphisms between these representatives also admit a description which is independent of~$i$.
\end{proof}

\noindent
Since $p \in \alpha \mathcal{O}$ we may always take the transversal $v_1,\ldots,v_p$ to be $0,1,\ldots,p-1$. For $p=3$ and $i=4$ the Quillen category has skeleton
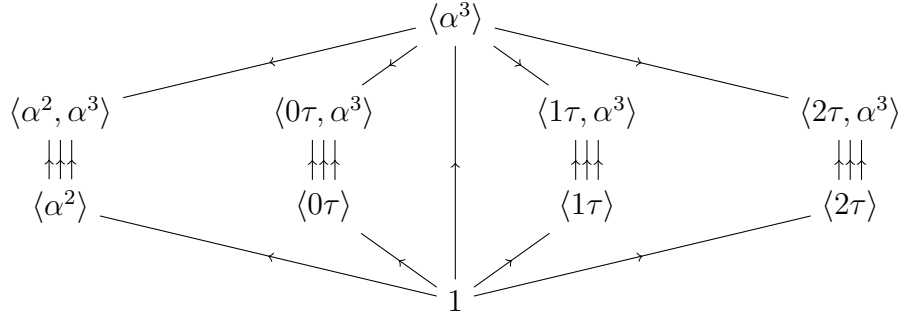
\[
\begin{tikzpicture}[x=1.75cm,y=1.25cm]
\node (I) at (0,0) {$1$};
\node (E0) at (-3,1) {$\left\langle \alpha^2 \right\rangle$};
\node (E1) at (-1,1) {$\left\langle 0\tau \right\rangle$};
\node (E2) at (1,1) {$\left\langle 1\tau \right\rangle$};
\node (E3) at (3,1) {$\left\langle 2\tau \right\rangle$};
\node (F0) at (-3,2) {$\left\langle \alpha^2, \alpha^3 \right\rangle$};
\node (F1) at (-1,2) {$\left\langle 0\tau, \alpha^3 \right\rangle$};
\node (F2) at (1,2) {$\left\langle 1\tau, \alpha^3 \right\rangle$};
\node (F3) at (3,2) {$\left\langle 2\tau, \alpha^3 \right\rangle$};
\node (G) at (0,3) {$\left\langle \alpha^3 \right\rangle$};
\foreach \a/\b in {I/G,I/E0,I/E1,I/E2,I/E3,G/F0,G/F1,G/F2,G/F3} \draw[postaction={decorate,sArrow}] (\a) -- (\b);
\foreach \a/\b in {E0/F0,E1/F1,E2/F2,E3/F3}
{
  \draw[postaction={decorate,sArrow}] (\a) -- (\b);
  \draw[transform canvas={xshift=-0.15cm},postaction={decorate,sArrow}] (\a) -- (\b);
  \draw[transform canvas={xshift=0.15cm},postaction={decorate,sArrow}] (\a) -- (\b);
}
\end{tikzpicture}
\]
where the three automorphisms of each rank two elementary abelian are omitted for clarity. Specifically, the three maps $\langle 2\tau \rangle \rightarrow \langle 2\tau,\alpha^3 \rangle$ are $2\tau \mapsto (2+\lambda \alpha^3)\tau$ for $\lambda = 0,1,2$; and three three autormorphisms of $\langle 2\tau, \alpha^3 \rangle$ fix $\alpha^3$ and act on $2\tau$ as one of these three maps.

\subsection{A more substantial example}
\label{subsection:G9}
Together with Leedham--Green, Newman and O'Brien, the first author studied the $3$-groups of coclass two in~\cite{E-LG-N-OB}\@. In particular they construct the skeleton groups in the four coclass trees (out of sixteen) whose branches have unbounded depth. Here we consider the skeleton groups in one of these unbounded depth trees: the tree associated to the pro-$3$-group which they denote as~$R$ (see their Theorem~4.2(a))\@.

We briefly recall the construction of the skeleton groups $R_{j-3,\gamma,m}$ from \cite[Sect.~5]{E-LG-N-OB}\@. Let $j \geq 7$. Let $K = \qq_3(\theta)$ be the ninth local cyclotomic number field, so $\theta$ is a root of $\Phi_9(X) = X^6 + X^3 + 1$. Let $\mathcal{O}$ be the ring of integers in~$K$; then $\mathcal{O} = \zz_3[\theta]$ is free as a $\zz_3$-module, with basis $1,\theta,\ldots,\theta^5$. Moreover, $\mathcal{O}$ is a local ring, with maximal ideal $\mathfrak{p} = (\theta-1)\mathcal{O}$. Observing that $(\theta-1)^6$ and $3$ are associates, one sees that $3\mathcal{O} = \mathfrak{p}^6$.

We now recall the twisting map $\mathfrak{p} \wedge \mathfrak{p} \rightarrow \mathcal{O}$, which we shall denote by~$\gamma_0$. Note however that in~\cite{E-LG-N-OB} it is called~$\vartheta$. It is the map
\[
\gamma_0(x \wedge y) = \sigma_2(x)\sigma_{-1}(y) - \sigma_{-1}(x) \sigma_2(y) \, ,
\]
where the automorphism $\sigma_r \in \operatorname{Gal}(K/\qq_3)$ is given by $\sigma_r(\theta) = \theta^r$. Lemma~5.1 of~\cite{E-LG-N-OB} shows that
\[
\gamma_0(\mathfrak{p}^i \wedge \mathfrak{p}^j) = \mathfrak{p}^{i+j+\eps} \quad \text{for} \quad \eps = \begin{cases} 3 & i \equiv j \pmod{3} \\ 2 & \text{otherwise} \end{cases} \, .
\]
Pick $j \geq 7$ and set $T = \mathfrak{p}^{j-3}$, $T_{\ell} = \mathfrak{p}^{j-3+\ell}$. Then $\gamma_0(T \wedge T) = T_j$, and $\gamma_0(T_j \wedge T) = T_k$ for
\[
k = \begin{cases} 2j & 3 \mid j \\ 2j-1 & 3 \nmid j \end{cases} \, .
\]
Now pick a unit $c \in \mathcal{O}^{\times}$ and set $\gamma = c \gamma_0$.
For any $m \in \{j,j+1,\ldots,k \}$ one defines $T_{j-3,\gamma,m}$ to be the group with underlying set $T/T_m$ and product
\[
(x + T_m) * (y + T_m) = \left( x + y + \frac12 \gamma(x \wedge y) \right) + T_m \, .
\]
Finally one sets $R_{j-3,\gamma,m} = T_{j-3,\gamma,m} \rtimes C$, where $C = \langle \tau \rangle$ has order~$9$ and acts on~$T$ via ${}^{\tau} v = \theta v$ for $v \in T$. Note that $\abs{R_{j-3,\gamma,m}} = 3^{m+2}$.

\begin{lemma}
\label{lemma:Rj}
Let $v,w \in T$.
\begin{enumerate}
\item \label{enum:Rj-1}
$(v+T_m)^r = rv+T_m$ in $T_{j-3,\gamma,m}$ for all $r \in \zz$.
\item \label{enum:Rj-2}
The order $3$ elements of $R_{j-3,\gamma,m}$ are:
\begin{itemize}
\bulit Elements of the form $(v + T_m)\tau^{3r}$, with $v \in T$ and $r \in \{1,2\}$;
\bulit Elements of the form $v + T_m$ with $v \in T_{m - 6}$.
\end{itemize}
\item \label{enum:Rj-3}
If $v+T_m$ has order~$3$ then $\gamma(v \wedge w) \in T_m$ for all $w \in T$. Hence $\Omega_1(T_{j-3,\gamma,m}) \leq Z(T_{j-3,\gamma,m})$.
\end{enumerate}
\end{lemma}

\begin{proof}
\enref{Rj-1}: Follows by induction, since $\gamma(v \wedge rv) = r \gamma(v \wedge v) = 0$. 
\item \enref{Rj-2}: Firstly, $[(v + T_m)\tau^3]^3 = (1+\theta^3 + \theta^6)v+T_m = 0$. Secondly: $(v+T_m)^3 = 3v+T_m$. This is zero for $v \in 3^{-1}T_m = T_{m-6}$.
\item \enref{Rj-3}: Suppose that $v \in T_{m-6}$ and $w \in T$. Then $\gamma(v \wedge w)$ lies in $T_{j+m-9+\eps}$, with $\eps \in \{2,3\}$. Since $\eps \geq 2$ and $j \geq 7$, this means that $\gamma(v \wedge w)$ lies in $T_m$.
\end{proof}

\begin{lemma}
\label{lemma:cclRj}
\begin{enumerate}
\item \label{enum:cclRj-0}
The orbit of $(v + T_m)\tau^3$ under conjugation by $T_{j-3,\gamma,m}$ is
\[
\left\{ (v' + T_m)\tau^3 \:\middle|\: v \in v + T_3 \right\} \, .
\]
\item \label{enum:cclRj-1}
$(v + T_m)\tau^3$ and $(w + T_m)\tau^3$ are conjugate in $R_{j-3,\gamma,m}$ if and only if $v+T_3$ and $w+T_3$ lie in the same orbit under the action of $C$ on $T/T_3$.
\ignore{
The conjugacy class of $(v + T_m)\tau^3$ in $R_{j-3,\gamma,m}$ consists of those $(w + T_m)\tau^3$ with
\[
w \in (v + T_3) \cup (\theta v + T_3) \cup (\theta^2 v + T_3) \, .
\]
These three residue classes coincide if $v \in T_2$, and are distinct otherwise.
}
\item \label{enum:cclRj-2}
The action of $R_{j-3,\gamma,m}$ on $T_{m-6}/T_m$ factors through $C$, and coincides wth the action of $C$ on $T/T_6$ via the isomorphism $v+T_6 \mapsto (\theta-1)^{m-6}v + T_m$.
\ignore{
If $v \in T_{m-6}$ then the conjugacy class of $v+T_m$ in $R_{j-3,\gamma,m}$ is
\[
\{\theta^r v+T_m \mid 0 \leq r \leq 8 \} \, .
\]
This has length $1$ for $v \in T_{m-1}$; length $3$ if $v \in T_{m-3} \setminus T_{m-1}$; and length nine otherwise.
}
\item \label{enum:cclRj-3}
$C_{T_{j-3,\gamma,m}}((v+T_m)\tau^3) = T_{m-3}/T_m$.
\end{enumerate}
\end{lemma}

\begin{proof}
\enref{cclRj-0}:
Since $T_m$ and the image of $\gamma$ lie in $T_j \subseteq T_7$ we have
\begin{align*}
{}^{(w+T_m)}[(v+T_m)\tau^3] & = [(w+T_m)*(v+T_m)*(-\theta^3 w+T_m)]\tau^3
\\ & \in \left( v+(1-\theta^3)w + T_7 \right) \tau^3 \, .
\end{align*}
Since $\mathfrak{p} = (\theta-1)\mathcal{O}$ and $3\mathcal{O} = \mathfrak{p}^6$ it follows that $(1-\theta^3)T = (1-\theta)^3 T = \mathfrak{p}^3 T = T_3$.
\item
So for each $v' \in v+T_3$ we find $w \in T$ with ${}^{(w+T_m)}[(v+T_m)\tau^3]  = (v''+T_m)\tau^3$ and $v'' \in v' + T_7$. If we now adjust $w$ by adding $u \in T_r$, then since $\gamma(T \wedge T_r) = T_{j+r-3+\varepsilon} \subseteq T_{r+6}$ we alter $v'$ by an element of $(1-\theta^3)u + T_{r+6}$. So if the error $v''-v'$ lies in  $T_{s+3}$, then with one correction we can reduce to an error in~$T_{s+6}$. Iterating reduces the error to an element of~$T_m$.
\item \enref{cclRj-1}:
Follows from \enref{cclRj-0}\@.
\ignore{
The description of the conjugacy class follows, since conjugation by $\tau^3$ has the same effect as conjugation by $-v+T_m \in T_{j-3,\gamma,m}$. Finally, the three residue classes coincide if $(\theta - 1)v \in T_3$, that is if $v \in T_2$.
} 
\qquad \enref{cclRj-2}:
The action factors by Lemma~\ref{lemma:Rj}\,\enref{Rj-3}\@. The second statement follows, since $C$ acts as multiplication by~$\theta$.
\ignore{
$v+T_m$ is central in $T_{j-3,\gamma,m}$, so the conjugacy class is as stated, and the length is $1$, $3$ or $9$\@. The length is one if $(\theta - 1)v \in T_m$, i.e.,\@ if $v \in T_{m-1}$; and if not then it is three if $(\theta^3-1)v \in T_m$. Since $3v \in T_m$ this is equivalent to $(\theta-1)^3v \in T_m$, i.e.,\@ $v \in T_{m-3}$. This also demonstrates~\enref{cclRj-3}\@.
} 
\item \enref{cclRj-3}: Follows from \enref{cclRj-2}, since $T_3/T_6$ is the subspace of $T/T_6$ consisting of elements fixed by $\tau^3$.
\end{proof}

\noindent
Let $d$ be the number of orbits\footnote{One easily verifies that $d=11$.} for the action of $C$~on $T/T_3$.
Pick $v_1,\ldots,v_d \in T$ such that the $v_i+T_3$ fom a set of orbit representatives for this action.

\begin{lemma}
\label{lem:Rjmaxels}
Every maximal elementary abelian subgroup of $R_{j-3,\gamma,m}$ is conjugate to precisely one of the following:
\begin{enumerate}
\item \label{enum:Rjmaxels-1}
$d$ rank four groups of the form $V_i = \langle (v_i + T_m)\tau^3 \rangle \times T_{m-3}/T_m$;
\item \label{enum:Rjmaxels-2}
$V_0 = T_{m-6}/T_m$ of rank six.
\end{enumerate}
If $U \leq V_i$ is not contained in $V_0$, then it is not conjugate to a subgroup of any other $V_j$.
\end{lemma}

\begin{proof}
Any elementary abelian outside $T_{m-6}/T_m$ must contain some element of the form $(v+T_m)\tau^3$ and is therefore contained in $\langle (v+T_m)\tau^3 \rangle \times C_{T_{j-3,\gamma,m}}((v+T_m)\tau^3)$, that is $\langle (v+T_m)\tau^3 \rangle \times T_{m-3}/T_m$. Since $m \geq j \geq 7$ and therefore $m-3 \geq 3$, no two of the rank four elementary abelians in~\enref{Rjmaxels-1} are conjugate. This argument also demonstrates the last part.
\end{proof}

\begin{theorem}
\label{thm:Rj}
Up to equivalence of categories, the Quillen category of the skeleton group $R_{j-3,\gamma,m}$ is independent of $j,\gamma,m$.
\end{theorem}

\begin{proof}
$V_0$ is a normal subgroup, and Lemma~\ref{lemma:cclRj}\,\enref{cclRj-2}\ describes the conjugation action. So by the last part of Lemma~\ref{lem:Rjmaxels} it suffices to show that if $U \leq V_i$ is not contained in~$V_0$, then the set of homomorphisms $U \rightarrow V_i$ lying in the Quillen category is independent of $j,\gamma,m$.

So $U = \langle (v + T_m)\tau^3 \rangle \times A$ for some $A \leq T_{m-3}/T_m$ and some $v \in v_i + T_{m-3}$. Consider conjugation by $(u+T_m)\tau^r$: by Lemma~\ref{lemma:cclRj} this can only send $(v + T_m)\tau^3$ to an element of~$V$ if $\theta^r v_i$ lies in $v_i + T_3$; and if $\theta^r v_i$ does lie there, then by adjusting $u$ we may send $(v + T_m)\tau^3$ to any element of the form $(v'+T_m)\tau^3$ with $v' \in v_i + T_{m-3}$. Moreover, the restriction to~$A$ of conjugation by $(u+T_m)\tau^r$ only depends on~$r$.
\end{proof}

\ignore{
Write $\alpha = \theta - 1$, so $T = \alpha^{j-3}\mathcal{O}$ and $T_{\ell} = \alpha{\ell+j-3}\mathcal{O}$. Since $3\mathcal{O} = \alpha^6 \mathcal{O}$ it follows that $\mathcal{O}/\alpha^3\mathcal{O}$ is a three-dimensional $\f[3]$-vector space with basis $1 + \alpha^3 \mathcal{O},\alpha + \alpha^3 \mathcal{O},\alpha^2 + \alpha^3 \mathcal{O}$. Hence $T/T_3$ has basis $\alpha^{j-3} + T_3$, $\alpha^{j-2} + T_3$ and $\alpha^{j-1} + T_3$, with $\alpha^{j-1}+T_3$ being a basis of $T_2/T_3$. Observe that the action of $\langle \tau^3 \rangle$ on $T/T_3$ induced by the action of $C$~on $T$ is $v + T_3 \mapsto (1 + \alpha) v + T_3$, which has eleven orbits:
\begin{itemize}
\item Each element of $T_2/T_3$ constitutes an orbit of length one (three orbits);
\item The elements $\alpha^{j-2}+T_3$ and $2\alpha^{j-2}+T_3$ of $T_1/T_3$ have orbits of length three (two orbits);
\item Each element of the form $\lambda \alpha^{j-3} + \mu \alpha^{j-1} + T_3$ with $\lambda \in \f[3]^{\times}$ and $\mu \in \f[3]$ has length three (six orbits)\@.
\end{itemize}
Let us label these representatives of the length three orbits as $x_1,\ldots,x_8 \in T$, where $x_7,x_8$ are the two elements of $T_1$. Taking $x_9=\alpha^{j-1}$, $x_{10} = 2\alpha^{j-1}$ and $x_{11} = 0 \in T$, we have representatives of all eleven orbits.
} 

\subsection{The generalized quaternion groups}

\noindent
Let $G$ be a finite group, and $k$ a field of characteristic~$p$. Write
\[
\bar{H}^*(G,k) = \lim_{E \in \Aa_p(G)} H^*(E,k) \, .
\]
Quillen~\cite[Th.\@ 6.2]{Quillen:spectrum} proved that the induced homomorphism $\phi_G \colon H^*(G,k) \rightarrow \bar{H}^*(G,k)$ induces a homeomorphism between prime ideal spectra.

Our result shows that if $G_r$ is a coclass family, then $\bar{H}^*(G_r,k)$ is independent of~$r$. However this does \emph{not} mean that the map $\phi_{G_r}$ is an isomorphism for large~$r$.
The (generalised) quaternion groups $Q_{2^n}$ ($n \geq 3$) provide a good example.

The quaternion groups form a coclass sequence. The mod-2 cohomology ring $H^*(Q_{2^n},\f[2])$ is well-known\footnote{To our knowledge the earliest references are \cite[p.~253-4]{CartanEilenberg} for the additive structure and \cite[p.~244]{Munkholm:Circle} for the ring structure. By 1987, Rusin~\cite[p.~316]{Rusin:metacyclic} could quote the result without needing to give a reference.}:
\begin{align*}
H^*(Q_8,\f[2]) & \cong \f[2][x,y,z]/(x^2+xy+y^2,x^2y+xy^2) \\
H^*(Q_{2^n},\f[2]) & \cong \f[2][x,y,z]/(x^2+xy,y^3) \quad (n \geq 4) \, ,
\end{align*}
with $x,y \in H^1$ and $z \in H^4$.
Since $H^1(G,\f[2]) = \Hom(G,\f[2])$ and all order two elements lie in the Frattini subgroup, it follows that $x,y \in \ker(\phi_{Q_{2^n}})$. In fact $\bar{H}^*(Q_{2^n},\f[2]) \cong \f[2][z]$, since $z$ restricts to the central $C_2$ as $t^4 \in H^*(C_2,\f[2]) \cong \f[2][t]$: see Rusin's construction of~$z$ as a top Stiefel--Whitney class~\cite[p.~316]{Rusin:metacyclic}\@. So both $H^*(Q_{2^n},\f[2])$ and $\bar{H}^*(Q_{2^n},\f[2])$ are constant for $n \geq 4$, but $\phi_{Q_{2^n}}$ is never injective.

In fact one can demonstrate that $\phi_{Q_{2^n}}$ is never injective without even knowing the cohomology of $Q_{2^n}$.
Recall that a class $x \in H^n(G,k)$ is called \emph{essential} if its restriction to every proper subgroup $H < G$ vanishes: so if $G$ is not elementary abelian, then every essential class lies in the kernel of $\phi_G$. Now, Adem and Karagueuzian showed~\cite{AdKa:Ess} that $H^*(G,\f)$ is Cohen--Macaulay and has non-zero essential elements if and only if $G$ is a $p$-group and all order $p$ elements are central. As $Q_{2^n}$ satisfies this group-theoretic condition, it follows that $\ker(\phi_{Q_{2^n}}) \neq 0$.




\end{document}